\documentclass [11pt,a4paper]{amsart}

\usepackage{amssymb} 
\usepackage{amsfonts} 
\usepackage{amsmath}
\usepackage{amsthm} 
\usepackage{tikz}
\usepackage[utf8]{inputenc}
\usepackage[all]{xy}
\SelectTips{cm}{}

\usepackage[pagebackref,
    ,pdfborder={0 0 0}
    ,urlcolor=black,a4paper,hypertexnames=false]{hyperref}
\hypersetup{pdfauthor={Clara L\"oh},pdftitle=Rank gradient vs.\ stable integral simplicial volume}

\usepackage{caption}

\DeclareMathOperator{\im}{im}
\DeclareMathOperator{\cost}{Cost}
\DeclareMathOperator{\rg}{rg}

\newtheorem{lem}{Lemma}[section]
\newtheorem{thm}[lem]{Theorem}

\newtheorem{cor}[lem]{Corollary} 

\theoremstyle{definition}
\newtheorem{defi}[lem]{Definition}

\newtheorem{question}[lem]{Question}

\newtheorem{rem}[lem]{Remark} 

\newcommand{\N}{\ensuremath {\mathbb{N}}}

\newcommand{\Q} {\ensuremath {\mathbb{Q}}}
\newcommand{\Z} {\ensuremath {\mathbb{Z}}}

 \makeatletter
 \newcommand\norm{\bBigg@{0.8}}

 \makeatother
 \newcommand{\inparens}[2][flex]{\csname #1l\endcsname(#2%
                                 \csname #1r\endcsname)\mathclose{}}
 \newcommand{\inangles}[2][flex]{\csname #1l\endcsname\langle#2%
                                 \csname #1r\endcsname\rangle\mathclose{}} 
 \newcommand{\innorm}[2][flex]{\csname #1l\endcsname|#2%
                                 \csname #1r\endcsname|\mathclose{}}
 \newcommand{\indnorm}[2][flex]{\csname #1l\endcsname\|#2%
                                 \csname #1r\endcsname\|\mathclose{}}
 \newcommand{\indnorml}[4][flex]{\csname #1l\endcsname\|#2%
                                 \csname #1r\endcsname\|_{#3}^{#4}\mathclose{}}

\newcommand{\sv}[2][flex]{\indnorm[#1]{#2}}%
\newcommand{\isv}[2][norm]{\indnorml[#1]{#2}{\Z}{}}
\newcommand{\isvc}[3][norm]{\indnorml[#1]{#2}{\Z}{#3}}
\newcommand{\pfcl}[2][flex]{\csname #1l\endcsname[#2%
                            \csname #1r\endcsname]}
\newcommand{\ifsv}[2][norm]{\csname #1l\endcsname\bracevert\!#2\!%
                            \csname #1r\endcsname\bracevert}
\newcommand{\stisv}[2][flex]{\indnorml[#1]{#2}{\Z}{\infty}}

\newcommand{\ucov}[1]{%
  \widetilde{#1}}

\title[Rank gradient vs.\ stable integral simplicial volume]%
      {Rank gradient\\ vs.\\ stable integral simplicial volume}

\author{Clara L\"oh}

\subjclass[2010]{57R19, 20E18, 20F65}

\keywords{stable integral simplicial volume, rank gradient}

\def\draftinfo{}
\date{\today.\ 
    This work was supported by the CRC~1085 \emph{Higher Invariants} 
    (Universit\"at Regensburg, funded by the DFG)\draftinfo}

\begin{document}

\begin{abstract}
  We observe that stable integral simplicial volume of closed manifolds
  gives an upper bound for the rank gradient of the corresponding
  fundamental groups.
\end{abstract}
\maketitle

\section{Introduction}

The \emph{residually finite view} on groups or spaces aims at understanding
groups and spaces through gradient invariants: If $I$ is an invariant of
groups, then we define the associated gradient invariant~$\widehat I$
for groups~$\Gamma$ by
\[ \widehat I(\Gamma) := \inf_{H \in F(\Gamma)} \frac{I(H)}{[\Gamma : H]},
\]
where $F(\Gamma)$ denotes the set of all finite index subgroups
of~$\Gamma$. For example, the rank gradient is the gradient invariant
associated with the minimal number of generators of groups
(Section~\ref{sec:rg}), originally introduced by
Lackenby~\cite{lackenby}. Further well-studied examples are the Betti number
gradient and the logarithmic torsion homology gradient.

Stable integral simplicial volume is the gradient invariant associated
with integral simplicial volume (Section~\ref{sec:stisv}). It is known
that stable integral simplicial volume yields upper bounds for Betti
number gradients and logarithmic torsion homology gradients~\cite[Theorem~1.6,
  Theorem~2.6]{FLPS}. 

In this note, we observe that stable integral simplcial volume also
gives an upper bound for the rank gradient of the corresponding
fundamental groups:

\begin{thm}\label{mainthm}
  Let $M$ be an oriented closed connected manifold with fundamental
  group~$\Gamma$ and let $\Gamma_* = (\Gamma_k)_{k \in \N}$ be a chain of 
  finite index subgroups of~$\Gamma$. Then
  \[ \rg (\Gamma, \Gamma_*)
  \leq 
  \isvc M{\Gamma_*}.
  \]
  In particular, 
  $ \rg \Gamma \leq 
  \stisv M.
  $
\end{thm}

The result even holds without any asymptotics (Lemma~\ref{lem:improved}),
but the gradient invariants seem to be the relevant invariants.

In particular, vanishing results for stable integral simplicial volume
imply corresponding vanishing results for the rank gradient. For
example, we obtain an alternative argument for the following rather special
case of a result by Lackenby~\cite[Theorem~1.2]{lackenby}:

\begin{cor}
  Let $\Gamma$ be a residually finite infinite amenable group that
  admits an oriented closed connected manifold as model of the
  classifying space~$K(\Gamma,1)$. Then 
  $\rg \Gamma = 0.
  $
\end{cor}
\begin{proof}
  Let $M$ be such a model of~$K(\Gamma,1)$. Then $\stisv M = 0$
  \cite[Theorem~1.10]{FLPS}. Hence, Theorem~\ref{mainthm} gives~$\rg
  \Gamma =0$.
\end{proof}

However, the bound in Theorem~\ref{mainthm} in general is far from
being sharp. For instance, if $\Gamma$ is the fundamental group of an
oriented closed connected surface~$M$ of genus~$g \in \N_{\geq 1}$,
then it is well known that $\rg \Gamma = b_1^{(2)} = 2 \cdot g
-2$~\cite{abertnikolov}\cite[Proposition~VI.9]{gaboriau}, but $\stisv
M = \sv M = 4 \cdot g - 4$~\cite{vbc,loehpagliantini}.  In more
positive words, Theorem~\ref{mainthm} shows that stable integral
simplicial volume is a geometric refinement of the rank gradient.

In contrast to the residual point of view, the \emph{dynamical view} on
groups or spaces aims at understanding groups and spaces through
actions on probability spaces. In the residually finite case, the
profinite completion provides a canonical link between the residually finite 
and the dynamical view.

For example, the $\Q$-Betti number gradients coincide with the
$L^2$-Betti numbers~\cite{lueckapprox}, which in turn admit a
dynamical description~\cite{gaboriaul2}. The dynamical sibling of
the rank gradient is cost~\cite{gaboriau,kechrismiller}, and the
cost (minus~$1$) of the profinite completion coincides with the rank
gradient~\cite{abertnikolov}. However, it remains an open problem to
decide whether the rank gradient and cost coincide for all residually
finite finitely generated groups.

The classical version of stable integral simplicial volume is Gromov's
simplicial volume; the dynamical version is integral foliated
simplicial volume~\cite{gromovmetric,mschmidt}, and integral foliated
simplicial volume with respect to the profinite completion coincides
with stable integral simplicial volume~\cite[Theorem~2.6]{FLPS}. Therefore, it is
natural to consider the following problem:

\begin{question}
  Let $M$ be an oriented closed connected aspherical manifold with
  (integral foliated) simplicial volume equal to~$0$.  Does this imply
  that $\cost \pi_1(M) = 1$\;? More generally, does (integral foliated)
  simplicial volume of~$M$ give a linear upper bound for~$\cost(\pi_1(M)) - 1$\;?
\end{question}

It should be noted that this is conjecturally trivial (but no direct
route is known): The Singer conjecture predicts that all $L^2$-Betti
numbers are~$0$, except possibly in the middle dimension. Moreover,
conjecturally, if the first $L^2$-Betti number is~$0$, the cost
equals~$1$. Both conjectures seem bold, but no counterexamples are
known. 

\subsection*{Organisation of this article}

We briefly review the rank gradient
(Section~\ref{sec:rg}) and stable integral simplicial volume
(Section~\ref{sec:stisv}). The (elementary) proof of Theorem~\ref{mainthm}
is given in Section~\ref{sec:proofmain}.

\section{Rank gradient}\label{sec:rg}

The rank gradient of a group is the gradient invariant associated with 
the rank: For a finitely generated group~$\Gamma$, we denote the rank
of~$\Gamma$, i.e., the minimal size of a generating set of~$\Gamma$,
by~$d(\Gamma)$.

\begin{defi}[rank gradient~\cite{lackenby}]
  Let $\Gamma$ be a finitely generated group and let $\Gamma_* =
  (\Gamma_k)_{k \in \N}$ be a (descending) chain of finite index subgroups
  of~$\Gamma$. Then the \emph{rank gradient of~$\Gamma$ with respect
    to~$\Gamma_*$} is defined as
  \[ \rg (\Gamma,\Gamma_*) 
     := \inf_{k \in \N} 
       \frac{d(\Gamma_k) - 1}{[\Gamma : \Gamma_k]}. 
  \]
  Moreover, the \emph{absolute rank gradient of~$\Gamma$} is defined
  as
  \[ \rg \Gamma := \inf_{H \in F(\Gamma)} \frac{d(H) -1}{[\Gamma : H]},\]
  where $F(\Gamma)$ denotes the set of all finite index subgroups of~$\Gamma$.
\end{defi}

If $\Gamma$ is a finitely generated group and $(\Gamma_k)_{k \in \N}$
is a chain of finite index subgroups, then the sequence
\[ \Bigl(\frac{d(\Gamma_k) - 1}{[\Gamma:\Gamma_k]} \Bigr)_{k \in \N}
\]
is non-increasing; therefore,
\[ \rg(\Gamma,\Gamma_*)
   = \inf_{k \in \N} \frac{d(\Gamma_k) - 1}{[\Gamma : \Gamma_k]}
   = \lim_{k \rightarrow \infty} \frac{d(\Gamma_k) -1}{[\Gamma : \Gamma_k]}.
\]

The rank gradient was originally introduced by
Lackenby~\cite{lackenby}. It is known that finitely generated
residually finite (infinite) amenable groups have trivial rank
gradient~\cite{lackenby,abertnikolov} and, more generally, that
residually finite groups that contain an infinite amenable normal
subgroup have trivial rank gradient~\cite{abertnikolov}. In contrast,
free groups of rank~$r \in \N_{\geq 1}$ have rank gradient~$r -
1$~\cite{lackenby} and fundamental groups of oriented closed connected
surfaces of genus~$g \in \N_{\geq 1}$ have rank gradient~$2 \cdot g -
2$~\cite{abertnikolov}\cite[Proposition~VI.9]{gaboriau}. Moreover,
some inheritance results are known for free products~\cite{abertjaikinzapirainnikolov},
certain free products with amalgamation, 
and certain HNN extensions~\cite{pappas}.
Further classes of groups with known rank gradient are certain Artin
groups and their relatives~\cite{karnikolovartin} and generalised
Thompson groups~\cite{kochloukova}.

In all known cases, the rank gradient coincides with the first
$L^2$-Betti number and is independent of the chosen chain of (normal) subgroups
(with trivial intersection), but it remains an open problem whether
this is always the case.

\section{Stable integral simplicial volume}\label{sec:stisv}

Similarly, stable integral simplicial volume is the gradient invariant 
associated with integral simplicial volume: The \emph{integral
  simplicial volume} of an oriented closed connected $n$-manifold~$N$ 
is defined by
\[ \isv N := \min 
   \biggl\{
   \sum_{j=1}^m |a_j|
   \biggm| \sum_{j=1}^m a_j \cdot \sigma_j \in C_n(N;\Z) 
           \text{ is a fundamental cycle of~$N$}
   \biggr\}.
\]

\begin{defi}[stable integral simplicial volume]
  Let $M$ be an oriented closed connected $n$-manifold with fundamental 
  group~$\Gamma$, let $\Gamma_* = (\Gamma_k)_{k \in \N}$ be a chain of 
  finite index subgroups of~$\Gamma$, and for~$k \in \N$ let~$M_k = \ucov M / \Gamma_k$ 
  be the covering manifold of~$M$ associated with the subgroup~$\Gamma_k$. 
  Then the \emph{stable integral simplicial volume of~$M$ with respect 
    to~$\Gamma_*$} is defined as
  \[ \isvc M {\Gamma_*} := \inf_{k \in \N} 
     \frac{\isv{M_k}}{[\Gamma : \Gamma_k]}.
  \]
  The \emph{stable integral simplicial volume of~$M$} is defined as
  \[ \stisv M := \inf_{H \in F(\Gamma)} \frac{\isv {\ucov M / H}}{[\Gamma :H]}.
  \]
\end{defi}

In the situation of the definition, the sequence
\[ \Bigl( \frac{\isv{M_k}}{[\Gamma : \Gamma_k]} \Bigr)_{k \in \N}
\]
is non-increasing; therefore,
\[ \isvc M {\Gamma_*} = \inf_{k \in \N} \frac{\isv{M_k}}{[\Gamma : \Gamma_k]}
   =\lim_{k \rightarrow \infty} 
     \frac{\isv{M_k}}{[\Gamma : \Gamma_k]}.
\]
As in the case for the rank gradient, it is unknown whether
stable integral simplicial volume is independent of the chosen
chain of subgroups (with trivial intersection) or not.

For aspherical oriented closed connected surfaces, for closed
hyperbolic \mbox{$3$-man}\-i\-folds, for closed Seifert manifolds with
infinite fundamental group, and for aspherical closed manifolds with
residually finite amenable fundamental group the stable integral
simplicial volume coincides with ordinary simplicial
volume~\cite{FLPS}. However, for closed hyperbolic manifolds of
dimension at least~$4$, stable integral simplicial volume is uniformly
bigger than ordinary simplicial volume~\cite{FFM}.

\section{Proof of Theorem~\ref{mainthm}}\label{sec:proofmain}

We will first give a simple geometric proof of Theorem~\ref{mainthm}
with a worse multiplicative constant
(Subsection~\ref{subsec:geomproof}) and then we will give a more
algebraic proof with the improved constant
(Subsection~\ref{subsec:algproof}). 

\subsection{A geometric proof}\label{subsec:geomproof}

We will prove the following version of Theorem~\ref{mainthm} (with a
slightly worse multiplicative constant): If $M$ is an oriented closed
connected $n$-manifold (with $n>0$) and if $(\Gamma_k)_{k \in \N}$ is
a chain of finite index subgroups of~$\Gamma$, then
\[ \rg(\Gamma,\Gamma_*) \leq n \cdot \isvc M {\Gamma_*}.
\]
This result follows from the following observation, by passing 
to finite coverings.

\begin{lem}[rank vs.\ integral simplicial volume]\label{lem}
  Let $M$ be an oriented closed connected $n$-manifold with
  fundamental group~$\Gamma$ (and~$n > 0$). Then
  \[ d(\Gamma) \leq n \cdot \isv M. 
  \]
\end{lem}
\begin{proof}
  Let $c = \sum_{j=1}^m a_j \cdot \sigma_j \in C_n(M;\Z)$ be an
  integral fundamental cycle of~$M$ in reduced form; without loss of
  generality, we may assume that all vertices of~$\sigma_1, \dots,
  \sigma_m$ are mapped to the basepoint of~$M$~\cite[Chapter~9.5]{tomdieck}.
  Out of the
  combinatorics of~$c$ we can then construct a connected
  CW-complex~$X_c$, a homology class~$\alpha_c \in H_n(X_c;\Z)$ and a
  continuous map~$f_c \colon X_c \longrightarrow M$ with
  \[ H_n(f_c;\Z)(\alpha_c) = [c] = [M]_\Z \in H_n(M;\Z)
  \]
  and
  \[ d \bigl(\pi_1(X_c)\bigr) 
     \leq n \cdot m.
  \]
  In detail, this is done as follows: We set
  \[ X_c := \bigl( \{1,\dots,m\} \times  \Delta^n\bigr) \bigm/\,\sim, 
  \]
  where the equivalence relation~$\sim$ is generated by the following gluing 
  conditions: For all~$j,j' \in \{1,\dots,m\}$ and all~$k,k' \in \{0,\dots,n\}$ 
  with~$\sigma_j \circ i_{k'} = \sigma_{j'} \circ i_{k'}$ we let
  \[ \forall_{t \in \Delta^{n-1}} \quad 
     \bigl(j, i_k(t)\bigr) \sim \bigl(j', i_{k'}(t)\bigr);
  \]
  here, $i_k \colon \Delta^{n-1} \longrightarrow \Delta^n$ denotes
  the inclusion of the $k$-th face of~$\Delta^n$. 
  We endow $X_c$ with the obvious cellular structure; moreover,
  without loss of generality we may assume that $X_c$ is connected
  (otherwise, we glue as many vertices as needed). 
  Moreover, looking 
  at the description of~$\pi_1(X_c)$ in terms of $1$- and $2$-cells, 
  shows that $d(\pi_1(X_c)) \leq n \cdot m$. 

  By construction, the maps~$\sigma_1, \dots, \sigma_m$ glue to give a
  well-defined continuous map~$f_c \colon X_c \longrightarrow M$.

  As next step, we construct~$\alpha_c$: For~$j \in \{1,\dots,m\}$ we
  consider the singular simplex~$\tau_j\colon \Delta^n \longrightarrow
  X_c$ induced by the $j$-inclusion~$\Delta^n \longrightarrow
  \{1,\dots,m\} \times \Delta^n$. Because $c$ is a cycle in~$M$ the
  gluing relation ensures that~$z_c := \sum_{j=1}^m a_j \cdot \tau_j
  \in C_n(X_c;\Z)$ is a cycle on~$X_c$. We write~$\alpha_c := [z_c]
  \in H_n(X_c;\Z)$ for the corresponding class. By construction, we
  then have~$H_n(f_c;\Z)(\alpha_c) = [c]$, as desired. This concludes 
  the construction of the model space~$X_c$ and its related objects. 

  We show that $\pi_1(f_c) \colon \pi_1(X_c) \longrightarrow
  \Gamma$ is surjective: Let $H := \im \pi_1(f_c) \subset \Gamma$, let
  $\pi \colon \overline M \longrightarrow M$ be the covering of~$M$
  associated with the subgroup~$H$, and let $\overline{f_c} \colon X_c
  \longrightarrow \overline M$ be the corresponding $\pi$-lift
  of~$f_c$. 
  Then 
  \[ H_n(\pi;\Z) \circ H_n(\overline{f_c};\Z) (\alpha_c) 
     = H_n(f_c;\Z) (\alpha_c) 
     = [M]_\Z.
  \]
  In particular, $[M]_\Z \in \im H_n(\pi;\Z)$, and so $|\deg \pi| = 1$ and~$H = \Gamma$. 

  Because $\pi_1(f_c)$ is surjective, we obtain
  \[ d(\Gamma) 
     \leq d\bigl(\pi_1(X_c)\bigr) 
     \leq n \cdot m
     \leq n \cdot |c|_1.
  \]
  Taking the minimum over all fundamental cycles of~$M$ proves the claim. 
\end{proof}

\subsection{A more algebraic proof}\label{subsec:algproof}

We will now prove an improved version of Lemma~\ref{lem}, which
implies Theorem~\ref{mainthm} by passing to finite coverings
and the infimum.

\begin{lem}[rank vs.\ integral simplicial volume, improved bound]\label{lem:improved}
  Let $M$ be an oriented closed connected $n$-manifold with
  fundamental group~$\Gamma$. Then
  \[ d(\Gamma) \leq \isv M. 
  \]
\end{lem}
\begin{proof}
  Let $c = \sum_{j=1}^m a_j \cdot \sigma_j \in C_n(M;\Z)$ be an
  integral fundamental cycle of~$M$ in reduced form.  We will now consider
  the corresponding situation on the universal covering: Let $\pi
  \colon \ucov M \longrightarrow M$ be the universal covering of~$M$
  and let $D \subset \ucov M$ be a set-theoretic fundamental domain of
  the deck transformation action of~$\Gamma$ on~$\ucov M$.  We then
  take the lift~$\widetilde c = \sum_{j = 1}^m a_j \cdot
  \widetilde \sigma_j \in C_n(\ucov M;\Z)$ of~$c$ such that for
  every~$j \in \{1,\dots,m\}$ the $0$-th vertex of~$\widetilde
  \sigma_j$ lies in~$D$.

  For each~$j \in \{1,\dots,m\}$ we let $g_j \in \Gamma$ be the group
  element that maps the $0$-th vertex of~$\widetilde \sigma_j \circ
  i_0$ (i.e., the first vertex of~$\widetilde \sigma_j$) to~$D$.  Then
  we set
  \[ S := \{g_1,\dots,g_m\}
     \quad\text{and}\quad
     H := \langle S \rangle_\Gamma \subset \Gamma.
  \]
  By construction~$d(H) \leq |S| = m \leq |c|_1$.

  We will now show
  that $H = \Gamma$ (which completes the proof of the lemma): Let
  $\pi_H \colon \ucov M \longrightarrow \ucov M / H =: \overline M$ be
  the upper covering associated with the subgroup~$H \subset \Gamma$
  and let $\overline c := C_n(\pi_H;\Z)(\widetilde c) \in
  C_n(\overline M;\Z)$. As first step we show that $\overline c$ is a cycle. 
  In
  \[ C_*(\overline M;\Z) \cong \Z \otimes_{\Z H} C_*(\ucov M;\Z)
  \]
  with the trivial $H$-action on~$\Z$, 
  we calculate (using~$g_1, \dots, g_m \in H$) that
  \begin{align*}
    \partial(\overline c)
    &  = \sum_{j=1}^m a_j \cdot g_j^{-1} \otimes g_j \cdot (\widetilde \sigma_j \circ i_0)
    + \sum_{j=1}^m \sum_{k=1}^n (-1)^k \cdot a_j \otimes
      \widetilde \sigma_j \circ i_k;\\
    & = \sum_{j=1}^m a_j \otimes g_j \cdot (\widetilde \sigma_j \circ i_0)
    + \sum_{j=1}^m \sum_{k=1}^n (-1)^k \cdot a_j \otimes
    \widetilde \sigma_j \circ i_k
    \\
    & = 1 \otimes \biggl(\sum_{j=1}^m a_j \cdot g_j \cdot (\widetilde \sigma_j \circ i_0)
    + \sum_{j=1}^m \sum_{k=1}^n (-1)^k \cdot a_j \cdot
    \widetilde \sigma_j \circ i_k\biggr)
    ;
  \end{align*}
  because the submodule of~$C_{n-1}(\ucov M;\Z)$ generated by
  simplices with $0$-th vertex in~$D$ is isomorphic (through
  the covering projection) with~$C_{n-1}(M;\Z)$ and because $\partial c =0$,
  it follows that the right hand term is~$0$. Therefore, $\overline c$
  is a cycle. 

  Let $p_H \colon \overline M \longrightarrow M$ be
  the lower covering associated with~$H \subset \Gamma$. Then
  \[ H_n(p_H;\Z)([\overline c]) = [c] = [M]_\Z,
  \]
  and so $[M]_\Z \in \im H_n(p_H;\Z)$. Therefore, $|\deg p_H|  =1$ and $H = \Gamma$.
\end{proof}

\begin{rem}
  Of course, the statement of Lemma~\ref{lem:improved} also
  generalises to the case of general spaces and homology classes
  in~$H_*(M;\Z)$ that do not lie in the image of the covering map of
  a finite (connected) covering of~$M$.
\end{rem}


\medskip
\vfill

\noindent
\emph{Clara L\"oh}\\[.5em]
  {\small
  \begin{tabular}{@{\qquad}l}
    Fakult\"at f\"ur Mathematik,
    Universit\"at Regensburg,
    93040 Regensburg\\
    \textsf{clara.loeh@mathematik.uni-r.de},\\
    \textsf{http://www.mathematik.uni-r.de/loeh}
  \end{tabular}}

\end{document}